\documentclass[a4paper]{amsart}

\usepackage{amsmath}
\usepackage{amssymb}
\usepackage{amsthm}
\usepackage[all]{xypic}
\usepackage{amsfonts}
\usepackage{fullpage}
\usepackage{verbatim}

\numberwithin{equation}{section}
\theoremstyle{plain}
\newtheorem{theorem}[equation]{Theorem}
\newtheorem{proposition}[equation]{Proposition}
\newtheorem{corollary}[equation]{Corollary}
\newtheorem{lemma}[equation]{Lemma}
\newtheorem{question}[equation]{Question}
\newtheorem{conjecture}[equation]{Conjecture}

\theoremstyle{definition}
\newtheorem{defn}[equation]{Definition}

\theoremstyle{remark}

\newcommand{\beq}{\begin{equation}}
\newcommand{\eeq}{\end{equation}}

\newcommand{\st}{\left\vert\right.}

\DeclareMathOperator{\Aut}{{Aut}}

\DeclareMathOperator{\GK}{GKdim}
\DeclareMathOperator{\gr}{gr}

\newcommand{\mc}{\mathcal}

\newcommand{\kk}{{\Bbbk}}

\newcommand{\ZZ}{{\mathbb Z}}

\newcommand{\PP}{{\mathbb P}}
\newcommand{\NN}{{\mathbb N}}

\newcommand{\sL}{\mc{L}}
\newcommand{\sM}{\mc{M}}

\DeclareMathOperator{\GKdim}{GKdim}

\DeclareMathOperator{\Pic}{Pic}

\DeclareMathOperator{\Supp}{Supp}
\newcommand{\ver}[1]{^{(#1)}}

\title{Algebras in which every subalgebra is noetherian}
\author{D. Rogalski,  S. J. Sierra, and J. T. Stafford}
\address{(Rogalski)
Department of Mathematics, UCSD, La Jolla, CA 92093-0112, USA. }
\email{drogalsk@math.ucsd.edu}
 \address{(Sierra) School of Mathematics,
University of Edinburgh, Edinburgh EH9 3JZ, Scotland.}
\email{s.sierra@ed.ac.uk}
\address{(Stafford) School of Mathematics,  The University of Manchester,   Manchester M13 9PL,
England.}
\email{Toby.Stafford@manchester.ac.uk}

\thanks{The first author is partially supported by NSF grant DMS-0900981.}
\thanks{The second author was supported by an NSF Postdoctoral Research
Fellowship, grant DMS-0802935.}
\thanks{The third
   author is a Royal Society Wolfson Research Merit Award Holder.}


\begin{document}
 \begin{abstract} We show that the twisted homogeneous coordinate rings of elliptic curves by infinite order automorphisms have the curious
property that \emph{every} subalgebra is both finitely generated and noetherian.  As a consequence, we show that a localisation of a generic
Skylanin algebra has the same property.
 \end{abstract}
 \maketitle

\section{Introduction}\label{INTRO} Throughout $\kk$ will denote  an algebraically closed field and all algebras will be $\kk$-algebras.
It is not hard to show that if $C$ is a commutative $\kk$-algebra with the property that every subalgebra is finitely generated (or noetherian),
then $C$ must be of Krull dimension at most one.  (See Section 3 for one possible proof.)  The aim of this note is to show that this property holds more generally in the noncommutative universe.

Before stating the result we need some notation. Let $X$ be a projective variety with invertible sheaf $\mathcal{M} $ and automorphism $\sigma$
and write $\mathcal{M}_n = \mathcal{M}\otimes \sigma^*(\mathcal{M})\cdots\otimes
(\sigma^{n-1})^*(\mathcal{M})$. 
Then the \emph{twisted  homogeneous
coordinate ring} of $X$ with respect to this data is the $\kk$-algebra
$B =B(X,\mathcal{M},\sigma)
=\bigoplus_{n \geq 0} H^0(X, \mathcal{M}_n)$, under a natural multiplication. 
This algebra is fundamental to the theory of noncommutative algebraic geometry; see for example \cite{ATV,St,SV}. In particular, if $S$ is the
3-dimensional Sklyanin algebra, as defined for example in \cite[Example~8.3]{SV},  then
$B(E,\mathcal{L},\sigma) =  S/gS$
for a central element $g\in S$ and some invertible sheaf $\mathcal{L}$ over an elliptic curve $E$.  
We note that $S$ is a graded ring that can be regarded as the ``coordinate ring of a noncommutative
 $\mathbb{P}^2$'' or as the ``coordinate ring of $\PP^2_{\rm nc}$''.  Under this analogy the ring
$A=(S[g^{-1}])_0$
 can be regarded as the noncommutative  (affine)   coordinate ring of
$\PP^2_{\rm nc}\smallsetminus E$.

With this notation, we can state the  main result of this paper.

\begin{theorem}\label{mainthm} {\rm (1)}   Let $B = B(E,\mathcal{M},\sigma)$, 
where $E$ is an elliptic curve, $\mathcal M$ is an invertible sheaf and $|\sigma|=\infty$.
Then every $\kk$-subalgebra of
$B$ is both finitely generated and noetherian.

{\rm (2)} Assume that $S$ is a $3$-dimensional Sklyanin algebra  that is not a
finite module over its centre and let $A=S[g^{-1}]_0$. Then every
$\kk$-subalgebra of
$A$ is both finitely generated and noetherian.
\end{theorem}

We note that the rings $A$ and $B$ of the theorem both have Gelfand-Kirillov dimension $2$.  

The main result is proved in the next section, while in the final section we raise some related questions.  

\section{The result}\label{LOCAL}

Apart from Sklyanin algebras there are two further classes of  algebras $S$ that are 
associated to elliptic curves  and to which the results of this paper apply; these algebras are the
generic examples of Artin-Schelter regular algebras of dimension three, as discussed in \cite[Section~8]{SV}. 
So we slightly change our notation to encompass all of them. Each ring is a graded $\kk$-algebra $S=\bigoplus_{n\geq 0} S_n$ that is a domain with a central element $g\in S_d$ such that 
$S/gS\cong B=B(E,\mathcal{L},\sigma)$ for  an elliptic curve $E$, ample invertible sheaf $\mathcal{L}$ and automorphism $\sigma $. The only significant difference between the three cases  is that $\mathcal{L}$ can have degree 1, 2 or 3; with the Sklyanin algebra occurring when $\deg(\mathcal{L})=3$. In these cases $d=6$, $4$ and $3$, respectively. 
These rings $S$ will be called  \emph{elliptic algebras}. We will assume throughout that the automorphism 
$\sigma$ has $|\sigma|=\infty; $ equivalently $B$ (and $S$) is infinite dimensional over its centre. The proofs of these assertions for $\deg(\mathcal{L})=2,3$ follow from \cite[Theorem~2]{ATV} and  \cite[Theorem~II]{ATV2}, with a more detailed description of the rings being given in \cite{ASh}; see
 \cite[(10.4) and (10.17)]{ASh} for a description of the central element $g$. The case 
  of $\deg(\mathcal{L})=1$ follows from \cite[Theorem~1.4]{Ste} with   the  centrality of $g$ given by  \cite[Lemma~3.3.6]{Ste2}.

It will also be convenient to work with the $d$-Veronese ring
$T=S^{(d)}=\bigoplus T_i$, where $T_i=S_{di}$. One should note that $T/gT $
may be identified with $B(E, \sM, \tau)$, where  $\sM=\sL_d = \sL\otimes \sigma^*(\sL) \otimes \cdots\otimes (\sigma^{d-1})^*(\sL)$
and $\tau=\sigma^d$.
Much of this paper is concerned with the   algebra
$S[g^{-1}]_0=T[g^{-1}]_0$, and so
for   any graded subalgebra  $R$ of $ T$ with $g\in R$, we write $R^o = R[g^{-1}]_0$.

 As a second way of presenting $R^o$ we have:

\begin{lemma}\label{prop:how to localise}
 Let $A=\bigoplus_{n\geq 0}A_n$ be an $\NN$-graded $\kk$-algebra  with  a central regular
element $g \in A_1$. Then
$ A[g^{-1}]_0 \cong A/(g-1)A.$
\end{lemma}
\begin{proof}
 Consider the vector space homomorphism  $ \phi:  A  \to A[g^{-1}]_0 $ defined
by sending $r \in A_n  \mapsto rg^{-n}.$
This is easily seen to be  a surjective ring homomorphism.

 We claim that $\ker \phi = (g-1)A$.
Clearly $g-1 \in \ker \phi$. Conversely,
suppose that  $r = \sum_{i=0}^N a_i \in \ker \phi$, with $a_i \in A_i$ and $a_N \neq 0$.  Then
$0=\sum a_i g^{-i} $, so
$a_N = - \sum_{i=0}^{N-1} a_i g^{N-i} $.
Thus
\[ r'= r + (g-1) (\sum_{i=0}^{N-1} a_i g^{N-i-1}) \ \in\ \ker \phi\]
 and $r'$ satisfies $\deg r'\leq N-1$.  By induction, $r' \in (g-1)A$, and so $r \in (g-1)A$.
\end{proof}

\begin{lemma}\label{prop:assoc graded of localisation}
  Suppose that  $A$ is an $\NN$-graded $\kk$-algebra  with   a central regular element  $g \in A_1$.
Let  $\phi:  A \to  A/(g-1)A$ be the canonical surjection and define a filtration $\Lambda$
on $ A/(g-1)A$ by setting $\Lambda_n = \phi(A_{\leq n})$.  Then $\gr_{\Lambda} (A/(g-1)A) \cong A/gA$.
\end{lemma}
\begin{proof}  If we filter $A$ by $\Lambda'_n = A_{\leq n}$, then the quotient map $A\to A/(g-1)A$ is a filtered surjection and so it induces a
graded surjection $ \pi: A  \to \gr_{\Lambda} (A/(g-1)A)$.
Finally,  $$(\ker \pi)_n =  \bigl(A_{\leq n-1} + (g-1)A\bigr)\cap A_n
= \bigl(A_{\leq n-1} + gA_{n-1}\bigr) \cap A_n = g A_{n-1}.$$  Thus $\ker \pi = gA$.
\end{proof}

Combining the last two lemmas gives:

\begin{corollary}\label{cor:graded of T0}
 Let $T=S^{(d)}$ be the $d$-Veronese of an elliptic  algebra,  as defined above.  Then the grading on
$T$ induces a filtration $\Lambda$ on $T^o \cong T/(g-1)T$ for which  $\gr_\Lambda
T^o\cong B(E, \sM, \tau)$.  \qed
\end{corollary}

We will see that $T^o$ has the  property that {\em all} its subalgebras are
noetherian.  This suggests the following definition.
 
\begin{defn}\label{def:supernoetherian}
 A $\kk$-algebra $A$ is {\em supernoetherian} if
\begin{itemize}
\item[(i)]  every  $\kk$-subalgebra of $A$ is  finitely generated;
\item[(ii)] every $\kk$-subalgebra of $A$ is noetherian.
\end{itemize}
An algebra  $A$ is called \emph{graded supernoetherian} if $A=\bigoplus_{i\geq 0}A_i$ is $\NN$-graded and (i) and (ii) hold for graded
subalgebras of $A$.
\end{defn}

We remark,  that condition (i) of Definition~\ref{def:supernoetherian} is equivalent to saying that $A$ satisfies the ascending chain condition
on subalgebras. Similar comments apply in the graded case.

As usual, the graded and ungraded versions of supernoetherianity are closely connected:

\begin{lemma}\label{graded}  Suppose that the algebra $A$ has a filtration $A=\bigcup_{i\geq 0}\Lambda^iA$ so that 
$\gr_{\Lambda}A=\bigoplus_{i \geq 0}
\Lambda^iA/\Lambda^{i-1}A$ is a  graded supernoetherian algebra. Then $A$ is a supernoetherian algebra. 
\end{lemma}
 
\begin{proof}  Given an ascending chain 
$R(1) \subseteq R(2) \subseteq \dots$ of $\kk$-subalgebras of $A$, 
we give them the filtrations $\Lambda^iR(n) = R(n) \cap \Lambda^i A$ induced from
$\Lambda.$ 
This ensures that each $\gr_{\Lambda}R(n)\subseteq \gr_{\Lambda}A$. The lemma now follows from two observations, both of which
 follow from the  argument of \cite[Proposition~1.6.7]{MR}:
 first if $R(1)\subsetneqq R(2)$ then
$\gr_{\Lambda}R(1)\subsetneqq \gr_{\Lambda}R(2)$ and, secondly, if $\gr_{\Lambda}R(1)$ is noetherian, then so is $R(1)$.
\end{proof}

By the last lemma, in order to prove that $T^o$ is supernoetherian, it suffices to consider
  $\gr_\Lambda T^o$, as we do next.

\begin{proposition}\label{prop:supernoetherian for B}
 Let  $\tau$ be an automorphism of  an elliptic curve $E$
of  infinite order, and let $\sM$ be an
invertible sheaf on $E$. Then every graded $\kk$-subalgebra of $B= B(E, \sM,
\tau)$ is noetherian, and $B$ has ACC on graded subalgebras.
\end{proposition}

\begin{proof} 
As noted in \cite[p.249]{AS}, the fact that  $\tau$ has infinite order means that
 it is given by translation by a point of infinite order under a group law on $E$, and so
$E$ has no point with a finite $\tau$-orbit.  Now recall that on the smooth elliptic curve $E$, an invertible sheaf $\mc{L}$ of degree $\leq 0$ has $H^0(E, \mc{L}) = 0$ unless $\mc{L} \cong \mc{O}_X$, 
in which case $H^0(E, \mc{L}) \cong \kk$.   Thus if $\mc{M}$ is of 
nonpositive degree (which on an elliptic curve is equivalent to being non-ample), 
then either $B \cong \kk$ or $B \cong \kk[x]$, where the degree of $x$ is equal to the order of $\sM$ in $\Pic E$.  In either case, the result is trivial.  
Therefore, for the rest of the proof we will assume that $\sM$ is ample.

   Note that if $B$ has ACC for \emph{finitely generated} graded subalgebras, then, necessarily,  all graded subalgebras
of $B$ will be finitely generated.    It therefore suffices to prove the proposition
for  finitely generated subalgebras.

Suppose that  $R$ is a finitely generated graded subalgebra of $B$, with $\kk \not= R$. 
 By \cite[Proposition~1.5]{AV},  the
 Gelfand-Kirillov dimension of $R$ satisfies $\GKdim R\leq \GKdim B=2$ and so, by \cite[Section~5]{SV},
$R$ has a graded ring of fractions
 $Q_{\gr} (R) $. As there,  
 we may write   $Q_{\gr}(R) = L[s^{\pm1}; \sigma] \subseteq \kk(E)[t^{\pm1}; \tau]=Q_{\gr}(B)$ for some field $L \subseteq \kk(E)$, $\sigma \in
\Aut_{\kk}(L)$, and homogeneous element $s$ of positive degree. 
 If $L = \kk$, then $R \subseteq \kk[s]$   and so is noetherian. So,
  suppose that $L$ has transcendence degree 1 over $\kk$,
  in which case we can write $L=\kk(X)$ for a (unique) smooth projective algebraic curve $X$
  \cite[Corollary~I.6.2]{Ha}. Now
    $s = t^kf$ for some $k\geq 1$ and $f\in \kk(E)$. As $\kk(E)$ is commutative, $\tau^k$ is equal to the  conjugation by $t^kf$ on
    $\kk(E)$ and so $\sigma={\tau^k}|_{ L}$.
    Let   $\sigma$ denote also the induced automorphism of $X$.  Since rational maps of smooth
projective curves are everywhere defined \cite[Proposition~I.6.8]{Ha}, there is a finite morphism $\pi:  E
\to X$ so that the diagram
\[\xymatrix{ E \ar[r]^{\tau^k} \ar[d]_{\pi} & E \ar[d]^{\pi} \\ X
\ar[r]_{\sigma} & X} \]
commutes.  As $\tau$ has infinite  order, so does $\sigma$ and so, by \cite[Theorem~5.6]{AS}, $R$ is
noetherian.

Since $E$ has no point with a finite $\tau$-orbit, $X$ has no point with a finite $\sigma$-orbit either.  This means that $X$ is also elliptic, and that  $\sigma$ is also a translation automorphism by a point of infinite order.   
  
It remains to prove the ascending chain condition.   Let
\beq\label{chain}R(1)
\subseteq R(2) \subseteq \cdots\eeq
 be a chain of finitely generated graded subalgebras of $B$; we may
assume that $R(1) \neq \kk$.
 Of course this chain induces an ascending chain of graded quotient rings $ Q_{\gr}R(1)\subseteq Q_{\gr}R(2)\subseteq \cdots\subseteq
Q_{\gr}(B)$.
  We first claim that the   $Q_{\gr}R(n)$ stabilise for $n \gg 0$.
   
Let $\Supp(R(n)) = \{ i \in \ZZ \st (Q_{\gr}R(n))_i \neq 0\}$.  Then
$\Supp(R(1)) \subseteq \Supp(R(2)) \subseteq \cdots$ is an ascending chain of
(nontrivial) subgroups of $\ZZ$.  Thus there exists $n_0 \in \NN$ such that   $ \Supp(R(n_0)) = \Supp(R(n)) =\ZZ r$,
 for $n\geq n_0$ and some natural number $r$. By replacing $B$ with the
Veronese $B^{(r)}$ and reindexing the $R(n)$, we may assume that $r=n_0 = 1$.
Thus, by taking $t$ to be any element of $(Q_{\gr}R(1))_1$, we have fields $\kk \subseteq K(1) \subseteq
K(2) \subseteq  \cdots \subseteq \kk(E)$  and compatible $\kk$-automorphisms
$\tau(n)$ of $K(n)$ so that
$Q_{\gr}(R(n)) \cong K(n)[t^{\pm1}; \tau(n)]$.   The maps $\tau(n)$ are all
compatible with the automorphism $\tau$ of $\kk(E)$.  Let $K =
\bigcup_n K(n) \subseteq
\kk(E)$ and  $\sigma = \tau|_K$.

Since $\kk$ is algebraically closed, $\kk(E)$ is finite dimensional over any subfield $F$ that properly contains $\kk$. Hence
$K$ has ACC on $\kk$-subfields and there exists $m_0 \in \NN$ so that $K(n) = K$ for $n
\geq m_0$.  By reindexing the $R(n)$ again, we may assume
that $m_0 =1$ and hence that  $Q_{\gr}R(n) =
  K[t^{\pm1}; \sigma]$ for all $n\geq 1$.
If $K = \kk$ then the $R(n)$ are all subalgebras of $\kk[t]$, which certainly has ACC on finitely generated
subalgebras.  So we may assume that $K$ has transcendence degree 1. As before, the fact that there are no finite $\tau$-orbits on $E$
forces  $K = \kk(X)$ for some elliptic curve $X$, with induced automorphism $\sigma: X \to X$ having no finite orbits. 
Since $Q_{\gr}R(n) = K[t^{\pm1}; \sigma]$,  it follows that  $\GKdim R(n)=2$ for each~$n$.

 Pick $j >0$  
so  that $R(1)_j \neq 0$ and consider the
$j^{\text{th}}$ Veronese subalgebras of the $R(n)$.
 For any $n$,  we have seen that $R(n)$ is noetherian.  By
\cite[Proposition~5.10(1)]{AZ1994} $R(n)\ver{j}$ is noetherian and so  finitely
generated as a $\kk$-algebra.  As  $\sigma$ has no finite orbits on $X$,  
 neither does $\sigma^j$.  By the  choice of $j$, we have $R(n)_j \supseteq R(1)_j \neq 0$.  Thus, by \cite[Lemma~2.3]{AS}, 
  the $R(n)\ver{j}$ all satisfy Hypothesis~2.15 of
\cite{AS}.  This is enough to ensure that we need take no further Veronese rings when applying \cite{AS} below.

 By  \cite[Proposition~6.4]{AS}, for each $n$ there are  integers $a_n \geq 0
$ and $b_n \geq 0$ so that
\[ \dim_\kk R(n)_{ji} = a_n i - b_n\qquad \text{
for all $i \gg 0$.}\]
For all $n$, we therefore have
\beq\label{FOO}
\dim_\kk R(n)_{ji} = a_n i - b_n \leq a_{n+1} i -b_{n+1} = \dim_\kk R(n+1)_{ji}  \qquad\text{  for all $i \gg 0$.}
\eeq

As $\GK R(1) =2$ certainly  $a_1 >0$, while  $a_1 \leq a_{2} \leq \cdots \leq j (\deg
\sM) = \dim_\kk B_{j}$, by Riemann-Roch.   Thus, there is some $n_1$ so that   $0< a_n = a_{n+1}
=a$, say, for all $n\geq n_1$.   By \eqref{FOO}, $b_n \geq b_{n+1} \geq 0$ for all $n \geq n_1$.  Thus, reindexing the $R(n)$ again, we may
assume that there are $a, b \geq 0$ so that $b_n = b$ and $a_n = a$ for all $n \geq 1$.  This says, therefore, that $R(n)\ver{j}/R(1)\ver{j}$ is
finite-dimensional for all $n \geq 1$.

We apply \cite[Theorem~5.9]{AS} to $R(n)\ver{j}$. This produces   an invertible sheaf
$\sL(n)$ on $X$, necessarily of degree $a=a_n$, so that $B(X, \sL(n), \sigma^j)$
contains and is a finitely generated right module over $R(n)\ver{j}$.
Since the construction of $\sL(n)$ depends only on the asymptotic behaviour of $R(n)\ver{j}$, and $R(n)\ver{j} /R(1)\ver{j}$ is
finite-dimensional, we must have $\sL(n) = \sL(1) $ for all $n \geq 1$.
 In particular,
each $R(n)\ver{j}$ is contained in $B(X, \sL(1), \sigma^j)$, which is, in turn,  a finitely
generated right module over $R(1)\ver{ j}$.

Set $R(\infty)=\bigcup R(n)\subseteq K[t; \sigma]$. From the first part of the proof of the proposition, $R(1)\ver j$ is noetherian and so
$R(\infty)\ver{j} \subseteq B(X, \sL(1), \sigma^j)$ is a finitely generated right
$R(1)\ver j$-module.  By \cite[Lemma~4.10(iii)]{AS}, $R(\infty)$ is a finitely generated right module over
$R(\infty)\ver j$ and therefore over $R(1)\ver{j}$.  Since the $R(n)$ are all $R(1)\ver j$-modules, the
chain \eqref{chain} stabilises.
 \end{proof}

Combining the earlier results we obtain:

\begin{theorem}\label{thm:supernoetherian}
{\rm(1)} Let $A$ be a $\kk$-algebra with a filtration $A=\bigcup_{n\geq 0}\Lambda_n$ such
that $\gr_\Lambda A \subseteq B(E, \sM, \tau)$ for some elliptic curve $E$,  invertible sheaf $\sM$ and 
 infinite order automorphism $\tau$.  Then $A$ is supernoetherian.

{\rm (2)} In particular,   $B(E, \sM, \tau)$ is supernoetherian.

{\rm (3)} Similarly, if $S$ is an elliptic algebra with $|\sigma|=\infty$,  then  $T^o=S^{(d)}/(g-1)S^{(d)}$  is supernoetherian.
\end{theorem}

\begin{proof} Part (1) follows immediately from
Proposition~\ref{prop:supernoetherian for B} combined with Lemma~\ref{graded}.
Parts (2) and (3) follow from Part (1) combined with Corollary~\ref{cor:graded of T0}.
\end{proof}

\begin{corollary}
Let $D=k(E)[t; \tau]$, where E is an elliptic curve and $\tau$ is an  infinite order automorphism.  Let $D$ 
be graded by degree in $t$ and let R be a finitely generated, connected graded subalgebra of $D$. Then $R$ is supernoetherian. 
\end{corollary}
\begin{proof}
Suppose that $R$ is generated in degrees $\leq d$.  Writing $R_i = V_i t^i$ for each $i$, let $\mc{N}$ be the subsheaf of the constant sheaf on $k(E)$ which is generated by the union of the sections in $V_0 = k, V_1, V_2, \dots, V_d$.   Then each sheaf 
$\mc{N}_n = \mc{N} \otimes \tau^*(\mc{N}) \otimes \dots \otimes (\tau^{n-1})^*(\mc{N})$ also has a natural induced embedding in the constant sheaf of 
rational functions, and since $\mc{O}_X \subseteq \mc{N}$, clearly  $\mc{N} \subseteq \mc{N}_n$ for all $n \geq 1$.   We then have  
\[
R \subseteq B(E, \mc{N}, \tau) \cong \bigoplus_{n \geq 0} H^0(E, \mc{N}_n) t^n \subseteq D.  
\]
By the theorem, $R$ is a subalgebra of a supernoetherian algebra, and so is itself supernoetherian.
\end{proof}

\section{Questions and Comments}
Various authors have studied commutative rings such that all subrings are noetherian
and, more generally,  extensions $R \subseteq S$
 of commutative  rings such that all intermediate subrings are noetherian; see, for example, \cite{Gi, Wa}.

We do not know how many supernoetherian algebras there are but, certainly,  they are very special.
For example, an infinite dimensional $\kk$-algebra $A$  is called \emph{just infinite} if every proper factor ring of $A$ is finite dimensional.
 It is easy to see that an infinite dimensional 
supernoetherian domain is just infinite. Indeed, suppose that   $I$ is a non-zero  ideal of infinite codimension in  a domain $A$ such   that 
$R=\kk+I$ is noetherian. Since $xA\subseteq R$  
for any $x\in I\smallsetminus \{0\}$, it follows that   $A$ is  a finitely generated $R$-module and hence
$A/I$ would be  a noetherian module over $\kk=R/I$, giving the required contradiction.  

The argument above also gives one method of verifying the claim concerning commutative algebras from the first paragraph of the introduction.  Indeed suppose that $R$ is a commutative $\kk$-algebra of Krull dimension $\geq 2$ such that all of its subalgebras are noetherian.
Then factoring out by some minimal prime of $R$ gives a  domain with the same properties.  But a just infinite commutative algebra
has Krull dimension $1$, so $R$ cannot be just infinite, contradicting the previous paragraph. 

The requirement that supernoetherian algebras be just infinite also implies rather easily that there are no further supernoetherian twisted homogeneous coordinate rings of curves.
A typical example  of this is   $D=B(\mathbb P^1,\mathcal{O}(1), \alpha)$
where $\alpha([a,b])=[a,a+b]$. An elementary exercise shows that   $D\cong \kk\{x,y\}/(xy-yx-x^2)$ and the
 ideal $I=xD$ satisfies $D/I\cong \kk[y]$; thus $R=\kk+xD$ is not noetherian.
 This  discussion also applies to  the Weyl algebra $A_1(\kk) =\kk\{x,\partial\}/(\partial x-x\partial-1)$ since
 $R\subset D\cong \kk\langle x, x^2\partial\rangle \subset A_1(\kk).$

More generally, suppose that  $\sigma$ is an automorphism of a 
projective variety $X$ and that $\sL$ is a $\sigma$-ample invertible sheaf on $X$  in
the  sense of \cite{AV} for which there exists 
    a proper $\sigma$-invariant subscheme $Y$ of $X$. (Except for the case of   translation by an 
    element of infinite order on an elliptic curve $X$, such a $Y$  exists whenever $X$ is a curve.)
 For any $\sigma$-ample $\mathcal
L$,  \cite[Lemma~4.4]{AS}   then produces an ideal $I$ of
  $B=B(X,\mathcal{L},\sigma)$ such that $B/I$ is, up to a finite dimensional
vector space,   the twisted homogeneous coordinate ring of $Y$. 
   Thus  $B/I$ will be infinite dimensional  and  $B$ will not be supernoetherian.

 In the above examples, the non-noetherian ring $R$ is also non-finitely generated and this suggests the following question that
  we cannot answer.  

  \begin{question} If ${\rm char}(\kk)=0$, does there exist a finitely generated, non-noetherian subalgebra of  $A_1(\kk)$
  or of the ring $D=B(\mathbb P^1,\mathcal{O}(1), \alpha)$?
 Note that   if  a counterexample $R\subseteq D$ exists  then \cite[Theorem~0.4]{AS} implies that  the associated
 graded ring of $R$ will be a  graded subalgebra of $D$ that is not finitely generated. 
  \end{question}
 
All supernoetherian algebras that we know of are either finite over a commutative domain of Krull dimension $\leq 1$ or algebras to which
Theorem~\ref{thm:supernoetherian} applies.
  It would be interesting to know if there are higher-dimensional examples, and we make the following conjecture.

 \begin{conjecture} Let $\sigma $ be an automorphism of an abelian variety $X$ that leaves invariant no
 proper subscheme. Then for any ample invertible sheaf $\mathcal{L}$ the ring
$B(X,\mathcal{L},\sigma)$ is  supernoetherian.
\end{conjecture}

 We believe that the techniques of \cite{RS,Si} can be used to answer this conjecture in the case of abelian surfaces, but the proof  would 
 require rewriting  \cite{Si} without the global noetherian hypothesis that was in place there.
 Note that, conjecturally, if $\sigma$ is an automorphism of a projective variety $X$ that has no proper invariant
subschemes, then $X$ must be an abelian variety  \cite[Conjecture~0.3]{RRZ2006}.
 
 Finally we note that the question of whether an algebra $R$  is supernoetherian is really only interesting when $R$ is a domain. For example,
suppose that $R$ is a prime noetherian $\kk$-algebra that is not a domain. 
 Then, by the Faith-Utumi Theorem \cite[Theorem~3.2.6]{MR},
 $R$ contains an equivalent matrix subring $M_n(S)$, where $S$ is a ring possibly without 1 and $n>1$. If $R\not=\kk$, then $S$ is infinite
dimensional and so the algebra $\kk+e_{12}S$ is neither finitely generated nor noetherian. Similarly, if $R$
 is a $\kk$-algebra with  an infinite dimensional nilradical $N(R)$,  then
 $\kk+N(R)$ will be neither finitely generated nor noetherian.

\bibliographystyle{amsalpha}
\providecommand{\bysame}{\leavevmode\hbox to3em{\hrulefill}\thinspace}
\providecommand{\MR}{\relax\ifhmode\unskip\space\fi MR }
\providecommand{\MRhref}[2]{%
  \href{http://www.ams.org/mathscinet-getitem?mr=#1}{#2}
}
\providecommand{\href}[2]{#2}

\end{document}